\documentclass[reqno]{amsart}
\usepackage{amsmath}
\usepackage{amsthm}
\usepackage{amsfonts}
\usepackage{amssymb}
\usepackage{graphicx}
\usepackage{mathrsfs}

\title{POINT DISTRIBUTIONS IN COMPACT METRIC SPACES, II}

\author{M.M. SKRIGANOV}

\address{St. Petersburg Department, Steklov Mathematical Institute, Russian Academy of Sciences}

\email{maksim88138813@mail.ru}

\dedicatory{Dedicated to the memory of Klaus Roth}



\numberwithin{equation}{section}

\newtheorem{theorem}{Theorem}[section]
\newtheorem{lemma}{Lemma}[section]
\theoremstyle{remark}
\newtheorem{remark}{Remark}[section]
\theoremstyle{remark}
\newtheorem{definition}{Definition}[section]

\def\dd{\mathrm{d}}

\def\Ee{\mathbb{E}}

\def\Rr{\mathbb{R}}

\def\AAA{\mathcal{A}}

\def\DDD{\mathcal{D}}
\def\EEE{\mathcal{E}}

\def\KKK{\mathcal{K}}

\def\MMM{\mathcal{M}}

\def\OOO{\mathcal{O}}
\def\PPP{\mathcal{P}}

\def\RRR{\mathcal{R}}

\DeclareMathOperator{\Lip}{Lip}
\DeclareMathOperator{\diam}{diam}
\DeclareMathOperator{\Diam}{Diam}

\renewcommand{\le}{\leqslant}
\renewcommand{\ge}{\geqslant}

\begin{document}

\begin{abstract}
We consider finite point subsets (distributions) in compact metric spaces.
In the case of general rectifiable metric spaces, non-trivial bounds for sums of distances between points of distributions and for discrepancies of distributions in metric balls are given (Theorem~\ref{thm1.1}).

We generalize Stolarsky's invariance principle to distance-invariant spaces (Theorem~\ref{thm2.1}).
For arbitrary metric spaces, we prove a probabilistic invariance principle (Theorem~\ref{thm3.1}).

Furthermore, we construct equal-measure partitions of general rectifiable compact metric spaces into parts of small average diameter (Theorem~\ref{thm4.1}). 

This version of the paper will be published in Mathematika 
\end{abstract}

\maketitle

\thispagestyle{empty}

%
%

\section{Introduction}\label{sec1}

Let $\MMM$ be a compact metric space with a fixed metric $\theta$ and a finite non-negative Borel measure~$\mu$, normalized by $\mu(\MMM)=1$.
For any metric $\rho$ on $\MMM$ and any $N$-point subset (distribution) $\DDD_N\subset \MMM$, we put
\begin{equation}
\rho[\DDD_N]=\sum_{x_1,x_2\in\DDD_N}\rho(x_1,x_2),
\label{eq1.1}
\end{equation}
and denote by $\langle\rho\rangle$ the average value of the metric~$\rho$, given by
\begin{equation}
\langle\rho\rangle=\int\!\!\!\!\int_{\MMM\times\MMM}\rho(y_1,y_2)\,\dd\mu (y_1)\,\dd\mu(y_2).
\label{eq1.2}
\end{equation}

We write $B_r(y)=\{x:\theta(x,y)<r\}$, $r\in T$, $y\in\MMM$, for the ball of radius $r$ centred at~$y$ and of volume $\mu(B_r(y))$.
Here $T=\{r:r=\theta(y_1,y_2),y_1,y_2\in\MMM\}$ is the set of radii, $T\subset[0,L]$, where $L=\sup\{r=\theta(y_1,y_2):y_1,y_2\in\MMM\}$ is the diameter of $\MMM$ in the original metric~$\theta$.

The \textit{local discrepancy} of a distribution $\DDD_N$ is defined by
\begin{align}
\Lambda[B_r(y),\DDD_N]
&
=\#(B_r(y)\cap \DDD_N)-N\mu(B_r(y))
\nonumber
\\
&
=\sum_{x\in\DDD_N}\Lambda(B_r(y),x),
\label{eq1.3}
\end{align}
where
\begin{equation}
\Lambda(B_r(y),x)=\chi(B_r(y),x)-\mu(B_r(y)),
\label{eq1.4}
\end{equation}
and $\chi(\EEE,x)$ is the characteristic function of a subset $\EEE\subset\MMM$.

The \textit{quadratic discrepancy} is defined by
\begin{equation}
\lambda_r[\DDD_N]=\int_{\MMM}\Lambda[B_r(y),\DDD_N]^2\,\dd\mu(y).
\label{eq1.5}
\end{equation}
This formula can be written as
\begin{equation}
\lambda_r[\DDD_N]=\sum_{y_1,y_2\in\DDD_N}\lambda_r(y_1,y_2),
\label{eq1.6}
\end{equation}
where
\begin{equation}
\lambda_r(y_1,y_2)=\int_{\MMM}\Lambda(B_r(y),y_1)\Lambda(B_r(y),y_2)\,\dd\mu(y).
\label{eq1.7}
\end{equation}

Let $\xi$ be a non-negative measure on the set $T$ of radii. 
We put
\begin{equation}
\lambda[\xi,\DDD_N]=\int_T\lambda_r[\DDD_N]\,\dd\xi(r)=\sum_{y_1,y_2\in\DDD_N}\lambda(\xi,y_1,y_2),
\label{eq1.8}
\end{equation}
where
\begin{equation}
\lambda(\xi,y_1,y_2)=\int_T\lambda_r(y_1,y_2)\,\dd\xi(r).
\label{eq1.9}
\end{equation}
It is clear that the integrals \eqref{eq1.8} and \eqref{eq1.9} converge if the measure $\xi$ is finite, while for special spaces $\MMM$, these integrals converge for much more
general measures~$\xi$; see~\cite{ref12}.

The quantity $\lambda[\xi,\DDD_N]^{1/2}$ is known as the $L_2$-discrepancy of a distribution $\DDD_N$ in balls $B_r(y)$, $r\in T$, $y\in\MMM$, with respect to the measures $\mu$ and~$\xi$.
In the present paper it is more convenient to deal with the quadratic discrepancy $\lambda[\xi,\DDD_N]$.

We introduce the extremal quantities
\begin{align}
\rho_N
&
=\sup_{\DDD_N}\rho[\DDD_N],
\label{eq1.10}
\\
\lambda_N(\xi)
&
=\inf_{\DDD_N}\lambda[\xi,\DDD_N],
\label{eq1.11}
\end{align}
with the supremum and infimum taken over all $N$-point distributions $\DDD_N\subset\MMM$.

The study of the quantities \eqref{eq1.10} and \eqref{eq1.11} falls within the subjects of the geometry of distances and discrepancy theory; see~\cite{ref2, ref5}.

In the present paper, we shall deduce non-trivial bounds for the quantities \eqref{eq1.10} and \eqref{eq1.11} under very general conditions on spaces~$\MMM$, metrics $\rho$ and measures $\mu$ and~$\xi$.
It is convenient to introduce the concept of $d$-rectifiable spaces, enabling us to compare the metric and measure on $\MMM$ with the Euclidean metric and Lebesgue measure on~$\Rr^d$.
The concept of rectifiability is well known in geometric measure theory; see~\cite{ref10}.
Here, this terminology is adapted for our purposes.

Recall that a map $f:\OOO\subset\Rr^d\to\MMM$ is Lipschitz if
\begin{equation}
\theta(f(Z_1),f(Z_2))\le c\Vert Z_1-Z_2\Vert,
\quad
Z_1,Z_2\in\OOO,
\label{eq1.12}
\end{equation}
with a positive constant~$c$, and the smallest such constant is called the Lipschitz constant of $f$ and denoted by $\Lip(f)$.
Here $\Vert\cdot\Vert$ denotes the Euclidean metric in~$\Rr^d$.

\begin{definition}\label{def1.1}
A compact metric space $\MMM$ with a metric $\theta$ and a measure $\mu$ is called $d$-rectifiable if there exist a measure $\nu$ on the $d$-dimensional unit cube $I^d=[0,1]^d$ that is absolutely continuous with respect to the $d$-dimensional Lebesgue measure on~$I^d$, a measurable subset $\OOO\subset I^d$, and an injective Lipschitz map $f:\OOO\to\MMM$, such that
\begin{itemize}
\item[(i)] $\mu(\MMM\setminus f(\OOO))=0$; and
\item[(ii)] $\mu(\EEE)=\nu(f^{-1}(\EEE\cap f(\OOO))$ for any $\mu$-measurable subset $\EEE\subset\MMM$.
\end{itemize}
\end{definition}

Since the map $f$ is injective, we can write
\begin{equation}
\nu(\KKK\cap\OOO)=\mu(f(\KKK\cap\OOO))
\label{eq1.13}
\end{equation}
for any measurable subset $\KKK\subset I^d$.
We can also assume that the measure $\nu$ is concentrated on $\OOO$ and $\nu(\OOO)=\mu(f(\OOO))=\mu(\MMM)=1$.

\begin{remark}\label{rmk1.1}
Simple examples of $d$-rectifiable spaces can be easily given.
Any smooth (or piece-wise smooth) compact $d$-dimensional manifold is $d$-rectifiable if in the local coordinates the metric satisfies \eqref{eq1.12}, and the measure is absolutely continuous with respect to the $d$-dimensional Lebesgue measure.
Particularly, any compact $d$-dimensional Riemannian manifold with the geodesic metric $\theta$ and the Riemannian measure $\mu$ is $d$-rectifiable.
In this case, it is known that condition \eqref{eq1.12} holds; see ~\cite[Chapter~I, Proposition 9.10]{ref8}.
On the other hand, the condition on the Riemannian measure is obvious because the metric tensor is continuous.
We refer the reader to ~\cite{ref10} for many more exotic examples of rectifiable spaces.
\end{remark}

In the present paper we shall prove the following theorem.

\begin{theorem}\label{thm1.1}
Suppose that a compact metric space $\MMM$, with a metric $\theta$ and a measure~$\mu$, is $d$-rectifiable.
Then the following hold:
\begin{itemize}
\item[(i)] If a metric $\rho$ on $\MMM$ satisfies the inequality
\begin{displaymath}
\rho(x_1,x_2)\le c_0\theta(x_1,x_2),
\end{displaymath}
with a constant $c_0>0$, then for each~$N$, we have
\begin{equation}
\rho_N\ge\langle\rho\rangle N^2-d2^{d-1}\Lip(f)c_0N^{1-1/d}.
\label{eq1.14}
\end{equation}
\item[(ii)] If a measure $\xi$ on the set $T$ of radii satisfies the condition
\begin{displaymath}
\xi([a,b))\le c_0(\xi)\vert a-b\vert,
\quad
a\le b
\mbox{ and }
a,b\in T,
\end{displaymath}
with a constant $c_0(\xi)>0$, then for each~$N$, we have
\begin{equation}
\lambda_N(\xi)\le d2^{d-3}\Lip(f)c_0(\xi)N^{1-1/d}.
\label{eq1.15}
\end{equation}
\end{itemize}

Here $\Lip(f)$ is the Lipschitz constant of the map $f$ in the definition of $d$-rectifiability of the space $\MMM$.
\end{theorem} 

Under such general assumptions one cannot expect that the bounds \eqref{eq1.14} and \eqref{eq1.15} are best possible.
One can give examples of $d$-dimensional manifolds and metrics where the bounds \eqref{eq1.14} and \eqref{eq1.15} can be improved. 
Consider, for example, the $d$-dimensional unit spheres $S^d=\{x\in\Rr^{d+1}:\Vert x\Vert=1\}$ with the geodesic (great circle) metric $\theta$ and the standard Lebesgue measure $\mu$ on ~$S^d$.
In this case, we have
\begin{equation}
\theta_N=\langle\theta\rangle N^2-\varepsilon_N,
\quad
\langle\theta\rangle=\frac{\pi}{2},
\label{eq1.16}
\end{equation}
where $\varepsilon_N=0$ for even $N$ and $\varepsilon_N=\pi/2$ for odd~$N$. We refer to ~\cite{ref6}
for the proof and detailed discussion of this relation.  

However, there are other examples where the order of the bounds \eqref{eq1.14} and \eqref{eq1.15} turns out to be sharp.
Instead of the geodesic metric~$\theta$, we can consider the chordal metric $\tau$, given by $\tau(x_1,x_2)=2\sin\frac{1}{2}\theta(x_1,x_2)$ on~$S^d$.
In this case, we have the two-sided bounds
\begin{equation}
\langle\tau\rangle N^2-CN^{1-1/d}<\tau_N<\langle\tau\rangle N^2-cN^{1-1/d}
\label{eq1.17}
\end{equation}
and
\begin{equation}
C_1N^{1-1/d}>\lambda_N(\xi^\natural)>c_1N^{1-1/d},
\label{eq1.18}
\end{equation}
with constants independent of $N$ and the measure $\dd\xi^\natural(r)=\sin r\,\dd r$ on the set of radii.
The left hand bounds in \eqref{eq1.17} and \eqref{eq1.18} were proved by Alexander~\cite{ref1} and Stolarsky~\cite{ref13}.
The right hand bounds in \eqref{eq1.17} and \eqref{eq1.18} were proved by Beck~\cite{ref4}; the proof involves Fourier analysis on $\Rr^{d+1}$.

The quantities $\tau_N$ and $\lambda_N(\xi^\natural)$ in the bounds \eqref{eq1.17} and \eqref{eq1.18} are not independent, and are related by the following identity.
For any $N$-point subset $\DDD_N\subset S^d$, we have
\begin{equation}
\alpha(S^d)\lambda[\xi^\natural,\DDD_N]+\tau[\DDD_N]=\langle\tau\rangle N^2,
\label{eq1.19}
\end{equation}
where $\alpha(S^d)>0$ is a constant independent of~$\DDD_N$. 
Particularly, for any $N$, we have
\begin{equation}
\alpha(S^d)\lambda_N(\xi^\natural)+\tau_N=\langle\tau\rangle N^2,
\label{eq1.20}
\end{equation}
and any bound for one of the quantities $\rho_N$ or $\lambda_N(\xi^\natural)$ implies a bound for the other one.

The identity \eqref{eq1.19} was established by Stolarsky~\cite{ref13} and is known in the literature as Stolarsky's invariance principle.
The original proof in \cite{ref13} was rather difficult.
It was simplified in the recent paper by Brauchart and Dick~\cite{ref7}, and further simplifications were given in the paper~\cite{ref6}.

Spheres as homogeneous spaces $S^d=SO(d+1)/SO(d)$ are the simplest examples of compact Riemannian symmetric spaces of rank one (two-point homogeneous spaces).
All such spaces are known.
Besides the spheres there are the real, complex, and quaternionic projective spaces and the octonionic projective plane; see, for example, \cite{ref8}.

By Theorem~\ref{thm1.1} the bounds \eqref{eq1.14} and \eqref{eq1.15} hold for all such spaces.
It turns out that the two-sided bounds of type \eqref{eq1.16} and \eqref{eq1.17} also hold for all these spaces and some classes of metrics on them.
The invariance principle \eqref{eq1.19} can also be generalized to projective spaces.
These results are intimately related with the geometry of projective spaces and Fourier analysis on homogeneous space.
The proof and detailed discussion of these results are recently given in our paper~\cite{ref12}.

In the present paper we use quite elementary methods going back to the papers by Alexander~\cite{ref1} and Stolarsky~\cite{ref13}.
Despite the simplicity, these methods turn out to be rather efficient.

In Section~\ref{sec2}, we introduce a class of symmetric difference metrics on metric spaces and give a generalization of Stolarsky's invariance principle to distance-invariant metric spaces (Theorem~\ref{thm2.1}).

In Section~\ref{sec3}, we give a probabilistic version of the invariance principle for arbitrary compact metric spaces (Theorem~\ref{thm3.1}).
With the help of this probabilistic invariance principle, we obtain the basic bounds for the quantities \eqref{eq1.10} and \eqref{eq1.11} in terms of equal measure partitions of a metric space (Theorem~\ref{thm3.2}). 
  
In Section~\ref{sec4}, we state our result on equal measure partitions of $d$-rectifiable compact metric spaces into parts of small average diameter (Theorem~\ref{thm4.1}).
Relying on this result and Theorem~\ref{thm3.2}, we complete the proof of Theorem~\ref{thm1.1}.

In Section~\ref{sec5}, we describe an explicit construction of equal measure partitions of $d$-rectifiable compact metric spaces into parts of small average diameter and prove Theorem~\ref{thm4.1}.

%
%

\section{The invariance principle for distance-invariant spaces}\label{sec2}

On an arbitrary compact metric space~$\MMM$, we introduce metrics associated with the fixed metric $\theta$ and measure $\mu$ by writing
\begin{equation}
\theta^\Delta(\xi,y_1,y_2)=\int_T\theta^\Delta_r(y_1,y_2)\,\dd\xi(r),
\label{eq2.1}
\end{equation}
where
\begin{equation}
\theta^\Delta_r(y_1,y_2)=\frac{1}{2}\mu(B_r(y_1)\Delta B_r(y_2)).
\label{eq2.2}
\end{equation}
Here
\begin{equation}
B_r(y_1)\Delta B_r(y_2)=(B_r(y_1)\cup B_r(y_2))\setminus(B_r(y_1)\cap B_r(y_2))
\label{eq2.3}
\end{equation}
denotes the symmetric difference of the balls $B_r(y_1)$ and $B_r(y_2)$.
Hence
\begin{align}
\theta^\Delta_r(y_1,y_2)
&
=\frac{1}{2}\int_{\MMM}\chi(B_r(y_1)\Delta B_r(y_2),y)\,\dd\mu(y)
\nonumber
\\
&
=\frac{1}{2}\int_{\MMM}\big(\chi(B_r(y_1),y)+\chi(B_r(y_2),y)-2\chi(B_r(y_1),y)\chi(B_r(y_2),y)\big)\,\dd\mu(y)
\nonumber
\\
&
=\frac{1}{2}\int_{\MMM}\vert\chi(B_r(y_1),y)-\chi(B_r(y_2),y)\vert\,\dd\mu(y).
\label{eq2.4}
\end{align}

For the average values of the metrics $\theta^\Delta(\xi)$ and~$\theta^\Delta_r$, we obtain
\begin{align}
\langle\theta^\Delta(\xi)\rangle
&
=\int_T\langle\theta^\Delta_r\rangle\,\dd\xi(r),
\label{eq2.5}
\\
\langle\theta^\Delta_r\rangle
&
=\int\!\!\!\!\int_{\MMM\times\MMM}\theta^\Delta_r(y_1,y_2)\,\dd\mu(y_1)\,\dd\mu(y_2)
\nonumber
\\
&
=\int_{\MMM}\big(\mu(B_r(y))-\mu(B_r(y))^2\big)\,\dd\mu(y),
\label{eq2.6}
\end{align}
where we have made use of the useful formula
\begin{equation}
\chi(B_r(y),x)=\chi (B_r(x),y)=\chi (r-\theta(x,y));
\label{eq2.7}
\end{equation}
here $\chi(t)$, $t\in \Rr$, is the characteristic function of the half-axis $[0,\infty)$.
The formula \eqref{eq2.7} holds in view of the symmetry of metric~$\theta$.

It is clear that the integrals \eqref{eq2.1} and \eqref{eq2.6} converge if the measure $\xi$ is finite while for special spaces~$\MMM$, these integrals converge for much more general measures~$\xi$; 
see [12, Lemma 2.1].

\begin{lemma}\label{lem2.1}
We have
\begin{equation}
\theta^\Delta(\xi,y_1,y_2)=\frac{1}{2}\int_{\MMM}\vert\sigma(\theta(y_1,y))-\sigma(\theta(y_2,y))\vert\,\dd\mu(y),
\label{eq2.8}
\end{equation}
where
\begin{equation}
\sigma(r)=\xi([r,L])=\int^L_r\dd\xi(t),
\quad
r\in T,
\label{eq2.9}
\end{equation}
and $L=\sup\{r:r\in T\}$ is the diameter of~$\MMM$.
\end{lemma}

\begin{proof}
For brevity, we write $\theta(y_1,y)=\theta_1$ and $\theta(y_2,y)=\theta_2$.
Using \eqref{eq2.1}, \eqref{eq2.4} and \eqref{eq2.7}, we obtain
\begin{align}
&
\theta^\Delta(\xi,y_1,y_2)
\nonumber
\\
&\quad
=\frac{1}{2}\int_{\MMM}\left(\int_T(\chi(r-\theta_1)+\chi(r-\theta_2)-2\chi(r-\theta_1)\chi(r-\theta_2))\,\dd\xi(r)\right)\dd\mu(y)
\nonumber
\\
&\quad
=\frac{1}{2}\int_{\MMM}\big(\sigma(\theta_1)+\sigma(\theta_2)-2\sigma(\max\{\theta_1,\theta_2\})\big)\,\dd\mu(y).
\label{eq2.10}
\end{align}
Since $\sigma$ is a non-increasing function, we have
\begin{align}
2\sigma(\max\{\theta_1,\theta_2\})
&
=2\min\{\sigma(\theta_1),\sigma(\theta_2\}
\nonumber
\\
&
=\sigma(\theta_1)+\sigma(\theta_2)-\vert\sigma(\theta_1)-\sigma(\theta_2)\vert.
\label{eq2.11}
\end{align}
Substituting  \eqref{eq2.11} into \eqref{eq2.10}, we obtain \eqref{eq2.8}.
\end{proof}

\begin{remark}\label{rmk2.1}
Using \eqref{eq2.8}, we can calculate the metric $\theta^\Delta(\xi)$ explicitly for special spaces $\MMM$ and measures~$\xi$.
For example, in the case of spheres $S^d$ and the special measure $\dd\xi^\natural(r)=\sin r\,\dd r$, one can easily find that the metric $\theta^\Delta(\xi^\natural)$ is proportional to the chordal metric~$\tau$, see \cite{ref6}.
For projective spaces and the specific measure $\xi^\natural$, the metric $\theta^\Delta(\xi^\natural)$ is proportional to the Fubini--Study metric, see ~\cite{ref12}.
\end{remark}

We next compare the metrics $\theta$ and $\theta^\Delta(\xi)$ on general metric spaces.
Note that using geometric features of spheres and projective spaces, the following result can be improved; see [12, Lemma 2.1].

\begin{lemma}\label{lem2.2}
If the measure $\xi$ satisfies the condition
\begin{equation}
\xi([a,b))\le c_0(\xi)\vert a-b\vert,
\quad
a\le b
\mbox{ and }
a,b\in T,
\label{eq2.12}
\end{equation}
with a constant $c_0(\xi)>0$, then we have the inequality
\begin{equation}
\theta^\Delta(\xi,y_1,y_2)\le\frac{1}{2} c_0(\xi)\theta(y_1,y_2).
\label{eq2.13}
\end{equation}
\end{lemma}

\begin{proof}
We use the same notation as in the proof of Lemma~\ref{lem2.1}.
Suppose first that $\theta_1\le\theta_2$.
Using \eqref{eq2.9}, \eqref{eq2.12} and the triangle inequality for the metric~$\theta$, we obtain
\begin{align}
\vert\sigma(\theta_1)-\sigma(\theta_2)\vert
&
=\xi([\theta_1,L])-\xi([\theta_2,L])
=\xi([\theta_1,\theta_2))
\le c_0(\theta_2-\theta_1)
\nonumber
\\
&
=c_0(\theta(y_2,y_1)-\theta(y_1,y))
\le c_0(\xi)\theta(y_1,y_2).
\label{eq2.14}
\end{align}
A similar inequality holds if $\theta_1>\theta_2$.
Substituting \eqref{eq2.14} into \eqref{eq2.8}, we obtain \eqref{eq2.13}.
\end{proof}

Consider the kernel \eqref{eq1.7}.
Substituting \eqref{eq1.4} into \eqref{eq1.7}, we obtain
\begin{align}
\lambda_r(y_1,y_2)
&
=\int_{\MMM}\big(\chi(B_r(y),y_1)\chi(B_r(y),y_2)-\mu(B_r(y))\chi(B_r(y),y_1)
\nonumber
\\
&\quad\qquad
-\mu(B_r(y))\chi(B_r(y),y_2)+\mu(B_r(y))^2\big)\,\dd\mu(y).
\label{eq2.15}
\end{align}
Comparing \eqref{eq2.4} and \eqref{eq2.15}, we see that
\begin{equation}
\lambda_r(y_1,y_2)+\theta^\Delta_r(y_1,y_2)=A^{(0)}_r+A^{(1)}_r(y_1)+A^{(1)}_r(y_2),
\label{eq2.16}
\end{equation}
where
\begin{equation}
A^{(0)}_r=\int_{\MMM}(B_r(y))^2\,\dd\mu(y)
\label{eq2.17}
\end{equation}
and
\begin{align}
A^{(1)}_r(x)
&
=\int_{\MMM}\left(\frac{1}{2}\chi(B_r(x),y)-\mu(B_r(y))\chi(B_r(y),x)\right)\dd\mu(y)
\nonumber
\\
&
=\frac{1}{2}\mu(B_r(x))-\int_{\MMM}\mu(B_r(y))\chi(B_r(y),x)\,\dd\mu(y)
\nonumber
\\
&
=\frac{1}{2}\mu(B_r(x))-\int_{\MMM}\mu(B_r(y))\chi(B_r(x),y)\,\dd\mu(y),
\label{eq2.18}
\end{align}
here we have used the formula \eqref{eq2.7}.

Let us consider these formulas in the following special case.
A metric space $\MMM$ is called \textit{distance-invariant} if, for each $r\in T$, the volume of ball $\mu(B_r(y))$ is independent of $y\in \MMM$; see~\cite{ref9}.
The typical examples of distance-invariant spaces are (finite or infinite) homogeneous spaces 
$\MMM=G/H$, where $G$ is a compact group, $H\vartriangleleft G$ is a closed subgroup, 
while $\theta$ and $\mu$ are respectively $G$-invariant metric and measure on~$\MMM$.

Numerous examples of distance-invariant spaces are known in algebraic combinatorics as distance-regular graphs and metric association schemes (on finite or infinite sets).
Such spaces satisfy the stronger condition that the volume of intersection $\mu(B_{r_1}(y_1)\cap B_{r_2}(y_2))$ of any two balls $B_{r_1}(y_1)$ and $B_{r_2}(y_2)$ depends only on $r_1,r_2$ and $r_3=\theta(y_1,y_2)$; see~\cite{ref3,ref9}.

For distance-invariant spaces, the integrals in \eqref{eq2.17} and \eqref{eq2.18} can be easily calculated, and we arrive at the following result. 

\begin{theorem}\label{thm2.1}
Let a compact metric space $\MMM$ with a metric $\theta$ and a measure $\mu$ be distance-invariant.
Then
\begin{equation}
\lambda_r(y_1,y_2)+\theta^\Delta_r(y_1,y_2)=\langle\theta^\Delta_r\rangle
\label{eq2.19}
\end{equation}
and
\begin{equation}
\lambda(\xi,y_1,y_2)+\theta^\Delta(\xi,y_1,y_2)=\langle\theta^\Delta(\xi)\rangle.
\label{eq2.20}
\end{equation}
Furthermore, if $\theta^\Delta[\xi,\DDD_N]$ and $\theta^\Delta_N(\xi)$ denote respectively the characteristics \eqref{eq1.1} and \eqref{eq1.10} with the metric
$\rho=\theta^\Delta(\xi)$, 
then
\begin{equation}
\lambda[\xi,\DDD_N]+\theta^\Delta[\xi,\DDD_N]=\langle\theta^\Delta(\xi)\rangle N^2
\label{eq2.21}
\end{equation}
and
\begin{equation}
\lambda_N(\xi)+\theta^\Delta_N(\xi)=\langle\theta^\Delta(\xi)\rangle N^2.
\label{eq2.22}
\end{equation}
Here $r\in T$ and $\DDD_N\subset\MMM$ is an arbitrary $N$-point subset.
The equalities \eqref{eq2.21} and \eqref{eq2.22} hold for any non-negative measure~$\xi$ such that the integrals \eqref{eq1.8}, \eqref{eq1.9}, \eqref{eq2.1} and \eqref{eq2.5} converge.
\end{theorem}

\begin{proof}
For brevity, we write $v_r=\mu(B_r(y))$.
By definition, $v_r$ is a constant independent of $y\in\MMM$, and \eqref{eq2.17} and \eqref{eq2.18} take the form
\begin{displaymath}
A^{(0)}_r=v^2_r
\quad\mbox{and}\quad
A^{(1)}_r(x)=\frac{1}{2}v_r-v^2_r.
\end{displaymath}
Hence the right side of \eqref{eq2.16} is equal to $v_r-v^2_r$.
On the other hand, the average value \eqref{eq2.6} is also equal to $v_r-v^2_r$.
This establishes \eqref{eq2.19}.
Integrating \eqref{eq2.19} over $r\in T$ with respect to the measure~$\xi$, we obtain \eqref{eq2.20}.
Summing \eqref{eq2.20} over $y_1,y_2\in\DDD_N$, we obtain \eqref{eq2.21}, and using \eqref{eq1.10} and \eqref{eq1.11}, we obtain \eqref{eq2.22}.
\end{proof}

Theorem~\ref{thm2.1} is a generalization of the invariance principle to arbitrary compact distance-invariant spaces.
For spheres $S^d$, the relation \eqref{eq2.21} implies Stolarsky's invariance principle \eqref{eq1.19}, since in this case the metrics $\theta^\Delta(\xi^\natural)$ and $\tau$ are proportional as mentioned earlier in Remark~\ref{rmk2.1}. 
Theorem~\ref{thm2.1} probably provides the most adequate explanation of the invariance principles.

%
%

\section{Equal-measure partitions and the probabilistic invariance principle}\label{sec3}

Is it possible to generalize invariance principles to arbitrary compact metric spaces?
At first glance the answer should be negative.
Nevertheless, a probabilistic generalization of such relations turns out to be possible.

First of all, we introduce some definitions and notation.
We consider an arbitrary compact metric space $\MMM$ with a fixed metric $\theta$ and a normalized measure~$\mu$.
Consider a partition $\RRR_{N}=\{V_i\}^N_1$ of $\MMM$ into $N$ measurable subsets $V_i\subset\MMM$, with
\begin{equation}
\mu\left(\MMM\setminus\bigcup^{N}_{i=1}V_i\right)=0,
\quad
\mu(V_i\cap V_j)=0,
\quad
i\ne j.
\label{eq3.1}
\end{equation}
We write $\diam(V,\rho)=\sup\{\rho(y_1,y_2):y_1,y_2\in V\}$ for the diameter of a subset $V\subset\MMM$ with respect to a metric $\rho$ on~$\MMM$.
For the partition \eqref{eq3.1}, we introduce the \textit{average diameter}
\begin{equation}
\Diam_1(\RRR_N,\rho )=\frac{1}{N}\sum^{N}_{i=1}\diam(V_i,\rho)
\label{eq3.2}
\end{equation}
and the \textit{maximum diameter}
\begin{equation}
\Diam_{\infty}(\RRR_N, \rho)=\max_{1\le i\le N}\diam(V_i,\rho).
\label{eq3.3}
\end{equation}
It is clear that
\begin{equation}
\Diam_1(\RRR_N,\rho)\le\Diam_{\infty}(\RRR_N,\rho),
\label{eq3.4}
\end{equation}
and that for two metrics $\rho_1$ and~$\rho$,
\begin{equation}
\left\{\begin{array}{l}
\Diam_1(\RRR_N,\rho_1)\le c_0\Diam_1(\RRR_N,\rho),
\\
\Diam_{\infty}(\RRR_N,\rho_1)\le c_0\Diam_\infty(\RRR_N,\rho),
\end{array}\right.
\label{eq3.5} 
\end{equation}
if $\rho_1(x,y)\le c_0\rho(x,y)$ for every $x,y\in\MMM$.

A partition $\RRR_N=\{V_i\}^N_1$ is an \textit{equal-measure partition} if all the subsets $V_i$ have equal measure $\mu (V_i)=N^{-1}$, $1\le i\le N$.

Suppose that an equal-measure partition $\RRR_N=\{V_i\}^N_1$ of the space $\MMM$ is given.
Introduce the probability space
\begin{equation}
\Omega_N
=\prod^{N}_{i=1}V_i
=\{X_N=(x_1,\ldots,x_N):x_i\in V_i,1\le i\le N\},
\label{eq3.6}
\end{equation}
with a probability measure
\begin{displaymath}
\omega_N=\prod^{N}_{i=1}\widetilde{\mu}_i,
\end{displaymath}
where $\widetilde{\mu}_i=N\mu\vert_{V_i}$.
Here $\mu\vert_{V_i}$ denotes the restriction of the measure $\mu$ to a subset $V_i\subset\MMM$.
We next write $\Ee_NF[\,\cdot\,]$ for the expectation of a random variable $F[X_N]$, $X_N\in\Omega_N$, and thus
\begin{align}
\Ee_NF[\,\cdot\,]
&
=\int_{\Omega_N}F[X_N]\,\dd\omega_N
\nonumber
\\
&
=N^N\int\!\ldots\!\int_{V_1\times\ldots\times V_N}F(x_1,\ldots,x_N)\,\dd\mu(x_1)\ldots\dd\mu(x_N).
\label{eq3.7}
\end{align}
Note that in the second equality, we have used the assumption that the subsets $V_i$ are of equal measure.

\begin{lemma}\label{lem3.1}
Let $F^{(1)}[X_N]$ and $F^{(2)}[X_N]$, $X_N=(x_1,\ldots,x_N)\in\Omega_N$, be random variables given by
\begin{equation}
F^{(1)}[X_N]=\sum_if (x_i)
\quad\mbox{and}\quad
F^{(2)}[X_N]=\sum_{i\ne j}f(x_i,x_j),
\label{eq3.8}
\end{equation}
where $f(y)$ and $f(y_1,y_2)$ are integrable functions on $\MMM$ and $\MMM\times\MMM$ respectively.
Then
\begin{align}
\Ee_NF^{(1)}[\,\cdot\,]
&
=N\int_{\MMM}f(y)\,\dd\mu(y)
\label{eq3.9}
\\
\intertext{and}
\Ee_NF^{(2)}[\,\cdot\,]
&
=N^2\int\!\!\!\!\int_{\MMM\times\MMM}f(y_1,y_2)\,\dd\mu(y_1)\,\dd\mu(y_2)
\nonumber
\\
&\quad\qquad
-N^2\sum^{N}_{i=1}\int\!\!\!\!\int_{V_i\times V_i}f(y_1,y_2)\,\dd\mu(y_1)\,\dd\mu(y_2).
\label{eq3.10}
\end{align}
\end{lemma}                      

\begin{proof}
Substituting the left equality in \eqref{eq3.8} into \eqref{eq3.7}, we obtain
\begin{displaymath}
\Ee_NF^{(1)}[\,\cdot\,]
=N\sum_i\int_{V_i}f(y)\,\dd\mu(y)
=N\int_{\MMM}f(y)\,\dd\mu(y).
\end{displaymath}
This proves \eqref{eq3.9}.

Substituting the right equality in \eqref{eq3.8} into \eqref{eq3.7}, we obtain
\begin{align}
&
\Ee_NF^{(2)}[\,\cdot\,]
=N^2\sum_{i\ne j}\int\!\!\!\!\int_{V_i\times V_j}f(y_1,y_2)\,\dd\mu(y_1)\,\dd\mu(y_2)
\nonumber
\\
&
=N^2\sum_{i,j}\int\!\!\!\!\int_{V_i\times V_j}f(y_1,y_2)\,\dd\mu (y_1)\,\dd\mu(y_2)-N^2\sum_i\int\!\!\!\!\int_{V_i\times V_i}f(y_1,y_2)\,\dd\mu(y_1)\,\dd\mu(y_2)
\nonumber
\\
&
=N^2\int\!\!\!\!\int_{\MMM\times\MMM}f(y_1,y_2)\,\dd\mu(y_1)\,\dd\mu(y_2)-N^2\sum_i\int\!\!\!\!\int_{V_i\times V_i}f(y_1,y_2)\,\dd\mu(y_1)\,\dd\mu(y_2).
\nonumber
\end{align}
This proves \eqref{eq3.10}.
\end{proof}

Elements $X_N=(x_1,\ldots,x_N)\in\Omega_N$ can be thought of as specific $N$-point distributions in the space~$\MMM$, and the corresponding sums of distances and discrepancies for $\DDD_N=X_N=\{x_1,\ldots,x_N\}\in\Omega_N$ can be thought of as random variables on the probability space~$\Omega_N$.
We put
\begin{align}
\rho[X_N]
&
=\sum_{i\ne j}\rho(x_i,x_j),
\label{eq3.11}
\\
\theta^\Delta_r[X_N]
&
=\sum_{i\ne j}\theta_r^\Delta(x_i,x_j),
\label{eq3.12}
\\
\theta^\Delta[\xi,X_N]
&
=\sum_{i\ne j}\theta^\Delta(\xi,x_i,x_j),
\label{eq3.13}
\end{align}
and
\begin{align}
\lambda_r[X_N]
&
=\sum_i\lambda_r(x_i,x_i)+\sum_{i\ne j}\lambda_r(x_i,x_j),
\label{eq3.14}
\\
\lambda[\xi,X_N]
&
=\sum_i \lambda(\xi, x_i,x_i)+\sum_{i\ne j}\lambda(\xi,x_i,x_j).
\label{eq3.15}
\end{align}

The probabilistic invariance principle can be stated as follows.

\begin{theorem}\label{thm3.1}
Let $\RRR_N$ be an equal-measure partition of a compact metric space~$\MMM$.
Then the expectations of the random variables \eqref{eq3.12}, \eqref{eq3.13}, \eqref{eq3.14} 
and \eqref{eq3.15} on the probability space $\Omega_N$ satisfy the relations
\begin{align}
\Ee_N\lambda_r [\,\cdot\,]+\Ee_N\theta^\Delta_r[\,\cdot\,]
&
=\langle\theta^\Delta_r\rangle N^2
\label{eq3.16}
\\
\intertext{and}
\Ee_N\lambda[\xi,\cdot\,]+\Ee_N\theta^\Delta[\xi,\cdot\,]
&
=\langle\theta^\Delta(\xi)\rangle N^2.  
\label{eq3.17}
\end{align}
\end{theorem}

\begin{proof}
Using \eqref{eq2.16} with $(y_1,y_2)=(x_i,x_j)$ and, summing over $x_i,x_j\in X_N$, we obtain
\begin{equation}
\lambda_r[X_N]+\theta^\Delta_r[X_N]=N^2A^{(0)}_r+2NA^{(1)}_r [X_N],
\label{eq3.18}
\end{equation}
where
\begin{displaymath}
A^{(1)}_r[X_N]=\sum_iA^{(1)}_r(x_i).
\end{displaymath}
We next calculate the expectation $\Ee_N$ of both sides in \eqref{eq3.18}.
Combining \eqref{eq3.9} with \eqref{eq2.17}, \eqref{eq2.18} and \eqref{eq2.6}, we find that
\begin{align}
&
\Ee_N\lambda_r[\,\cdot\,]+\Ee_N\theta^\Delta_r[\,\cdot\,]
\nonumber
\\
&
=N^2A^{(0)}_r+2\Ee_NA^{(1)}_r[\,\cdot\,]
=N^2A^{(0)}_r+2N^2\int_{\MMM} A^{(1)}_r (y)\,\dd\mu(y)
\nonumber
\\
&
=N^2\int_{\MMM}\mu(B_r(y))^2\,\dd\mu(y)+N^2\int_{\MMM}\mu(B_r(y))\,\dd\mu(y)-2N^2 \int_{\MMM}(B_r(y))^2\,\dd\mu(y)
\nonumber
\\
&
=N^2\int_{\MMM}\big(\mu(B_r(y))-\mu(B_r(y))^2\big)\,\dd\mu(y)
=\langle\theta^\Delta_r\rangle N^2.
\nonumber
\end{align}
This establishes \eqref{eq3.16}. 

Integrating \eqref{eq3.16} over $r\in T$ with respect to the measure~$\xi$, we obtain \eqref{eq3.17}.
\end{proof}

We wish to evaluate the expectation \eqref{eq3.7} of the random variable \eqref{eq3.11} for an arbitrary metric~$\rho$.

\begin{lemma}\label{lem3.2}
For any equal-measure partition $\RRR_N$ of the space $\MMM$ and any arbitrary metric $\rho$ on~$\MMM$, we have
\begin{equation}
\Ee_N\rho[\,\cdot\,]
\ge\langle\rho\rangle N^2-\Diam_1(\RRR_N,\rho)N
\ge\langle\rho\rangle N^2-\Diam_{\infty}(\RRR_N,\rho)N.
\label{eq3.19}
\end{equation}
\end{lemma}

\begin{proof}
Applying \eqref{eq3.10} to the random variable \eqref{eq3.11}, we obtain
\begin{displaymath}
\Ee_N\rho[\,\cdot\,]=\langle\rho\rangle N^2-Q_N(\rho)N^2,
\end{displaymath}
where
\begin{align}
Q_N(\rho)
&
=\sum_i\int\!\!\!\!\int_{V_i\times V_i}\rho(y_1,y_2)\,\dd\mu(y_1)\,\dd\mu(y_2)
\le N^{-2}\sum_i\diam(V_i,\rho)
\nonumber
\\
&
=\Diam_1(\RRR_N,\rho)N^{-1}
\le\Diam_{\infty}(\RRR_N,\rho)N^{-1},
\nonumber
\end{align}
and \eqref{eq3.19} follows.
\end{proof}

Distributions $X_N\in\Omega_N$ form a subset in the set of all $N$-point distributions $\DDD_N\subset\MMM$.
Hence
\begin{equation}
\rho_N\ge\Ee_N\rho[\,\cdot\,]
\quad\mbox{and}\quad
\lambda_N(\xi)\le\Ee_N\lambda[\xi,\cdot\,]. 
\label{eq3.20}
\end{equation}
Using these inequalities together with Lemma~\ref{lem3.2} and Theorem~\ref{thm3.1}, we arrive at the following basic bounds.

\begin{theorem}\label{thm3.2}
Let $\RRR_N$ be an equal-measure partition of a compact metric space $\MMM$ with a metric $\theta$ and a measure~$\mu$. 
Then the following hold:
\begin{itemize}
\item[(i)] If a metric $\rho$ on $\MMM$ satisfies the inequality
\begin{equation}
\rho(x_1,x_2)\le c_0\theta(x_1,x_2)
\label{eq3.21}
\end{equation}
with a constant $c_0>0$, then
\begin{align}
\rho_N
&
\ge\langle\rho\rangle N^2-c_0\Diam_1(\RRR_N,\theta)N
\nonumber
\\
&
\ge\langle\rho\rangle N^2-c_0\Diam_{\infty}(\RRR_N,\theta)N.
\label{eq3.22}
\end{align}
\item[(ii)] If the metric $\theta^\Delta(\xi)$ satisfies the inequality
\begin{equation}
\theta^\Delta(\xi,x_1,x_2)\le c_0\theta(x_1,x_2)
\label{eq3.23}
\end{equation}
with a constant $c_0>0$, then
\begin{align}
\theta^\Delta_N(\xi)
&
\ge\langle\theta^\Delta(\xi)\rangle N^2-c_0\Diam_1(\RRR_N,\theta)N
\nonumber
\\
&
\ge\langle\theta^\Delta(\xi)\rangle N^2-c_0\Diam_{\infty}(\RRR_N,\theta)N.
\label{eq3.24}
\end{align}
and
\begin{equation}
\lambda_N(\xi)
\le c_0\Diam_1(\RRR_N,\theta)N
\le c_0\Diam_{\infty}(\RRR_N,\theta)N.
\label{eq3.25}
\end{equation}
\end{itemize}
\end{theorem}

\begin{proof}
Comparing the left hand inequality in \eqref{eq3.20} with \eqref{eq3.19} and using \eqref{eq3.5}, we obtain \eqref{eq3.22}.
The bound \eqref{eq3.24} coincides with \eqref{eq3.22} written for the metric $\theta^\Delta(\xi)$.
The invariance principle \eqref{eq3.17} together with \eqref{eq3.20} gives
\begin{equation}
\lambda_N(\xi)\le\langle\theta^\Delta(\xi)\rangle N^2-\theta^\Delta_N(\xi).
\label{eq3.26}
\end{equation}
Substituting \eqref{eq3.24} into \eqref{eq3.26}, we obtain \eqref{eq3.25}.
\end{proof}

%
%

\section{Equal-measure partitions of small average diameter}\label{sec4}

In the next section we shall prove the following general result.

\begin{theorem}\label{thm4.1}
Suppose that a compact metric space $\MMM$, with a metric $\theta$ and a measure~$\mu$, is $d$-rectifiable.
Then for each $N$, there exists an equal measure partition $\RRR_N$ of the space $\MMM$ such that
\begin{equation}
\Diam_1(\RRR_N,\theta)\le d2^{d-1}\Lip(f)N^{-1/d},
\label{eq4.1}
\end{equation}
where $\Lip(f)$ is the Lipschitz constant of the map $f$ in the definition of $d$-rectifiability of the space $\MMM$.
\end{theorem}

Comparing Theorems \ref{thm3.2} and~\ref{thm4.1}, we arrive immediately at the following result.

\begin{theorem}\label{thm4.2}
Suppose that a compact metric space $\MMM$, with a metric $\theta$ and a measure~$\mu$, is $d$-rectifiable.
Write $C=d2^{d-1}\Lip(f)$.
Then the following hold:
\begin{itemize}
\item[(i)] If a metric $\rho$ on $\MMM$ satisfies the inequality
\begin{equation}
\rho(x_1,x_2)\le c_0\theta(x_1,x_2)
\label{eq4.2}
\end{equation}
with a constant $c_0>0$, then
\begin{equation}
\rho_N\ge\langle\rho\rangle N^2-c_0CN^{1-1/d}.
\label{eq4.3}
\end{equation}
\item[(ii)] If the metric $\theta^\Delta(\xi)$ satisfies the inequality
\begin{equation}
\theta^\Delta(\xi,x_1,x_2)\le c_0\theta(x_1,x_2)
\label{eq4.4}
\end{equation}
with a constant $c_0>0$, then
\begin{equation}
\theta^\Delta_N(\xi)\ge\langle\theta^\Delta(\xi)\rangle N^2-c_0CN^{1-1/d}
\label{eq4.5}
\end{equation}
and
\begin{equation}
\lambda_N(\xi)\le c_0CN^{1-1/d}.
\label{eq4.6}
\end{equation}
\end{itemize}
\end{theorem}

\begin{remark}\label{rmk4.1}
Notice that all the statements of Theorem~\ref{thm4.1} hold for an arbitrary compact $d$-dimensional Riemannian manifold with the geodesic metric $\theta$ and the Riemannian measure~$\mu$, since such manifolds are $d$-rectifiable; see Remark~\ref{rmk1.1}.
Applications of Theorem~\ref{thm4.1} to compact Riemannian symmetric spaces of rank one (two-point homogeneous spaces) are considered in our paper ~\cite{ref12} .
\end{remark}

Theorem~\ref{thm1.1} is a direct corollary of Theorem~\ref{thm4.2}. 

\begin{proof}[Proof of Theorem~\ref{thm1.1}]
The statement (i) is a paraphrase of the statement (i) of Theorem~\ref{thm4.1}.
On the other hand, by Lemma~\ref{lem2.2}, the inequality \eqref{eq4.4} is satisfied, and \eqref{eq4.6} implies \eqref{eq1.15}.
\end{proof}

For spheres $S^d$, equal measure partitions $\RRR_N$ can be constructed to satisfy
\begin{equation}
\Diam_{\infty}(\RRR_N,\theta)\le c(d)N^{-1/d},
\label{eq4.7}
\end{equation}
with a constant $c(d)$ independent of~$N$. The bound (4.7) has been used in many papers
on point distributions on spheres $S^d$, see, for example, \cite{ref1,ref4,ref5,ref11}.
Its detailed proof can be found in \cite{ref11}.

This last paper \cite{ref11} is the first in the literature that describes and quantifies the equal-area partitioning of the sphere with small diameter.
It is clear that the bound \eqref{eq4.7} is stronger than \eqref{eq4.1}; \textit{cf.} \eqref{eq3.4}.
However, the construction of equal-measure partitions $\RRR_N$ satisfying \eqref{eq4.7} depends significantly on the geometry of spheres $S^d$ as smooth submanifolds in
$\Rr^{d+1}$, while the bound \eqref{eq4.1} holds for arbitrary compact $d$-rectifiable metric spaces.
Furthermore, the bound \eqref{eq4.1} is not very sensitive to variation of metric and measure on a given space~$\MMM$.
In any case, the bound \eqref{eq4.1} suffices to prove Theorem~\ref{thm1.1}.

%
%

\section{Construction of equal-measure partitions of small average diameter}\label{sec5}

The proof of Theorem~\ref{thm4.1} relies on three auxiliary results.
Lemma~\ref{lem5.1} is trivial but uses at each stage our inductive construction.
Our construction of partitions is described in Lemma~\ref{lem5.2} for a special case of a measure concentrated on the $d$-dimensional unit cube.
The bound \eqref{eq4.1} for such equal measure partitions of the unit cube is given 
in Lemma~\ref{lem5.3}.
Once these partitions of the unit cube are constructed, the proof of Theorem~\ref{thm4.1} can be easily completed in view of Definition~\ref{def1.1}.

Let $\nu_0$ be a finite non-negative measure on the unit interval $I=[0,1]$.
Suppose that the measure $\nu_0$ is continuous, \textit{i.e.} it does not have a discrete component.
Then the distribution function $\varphi(z)=\nu_0([0,z])$, $z\in I$, is continuous, non-decreasing, and satisfies $\varphi(0)=0$ and $\varphi(1)=\nu_0(I)$.
Furthermore, there is a one-to-one correspondence between functions with such properties and finite continuous measures on~$I$.

Since the graph of $\varphi$ can have horizontal parts, we define the inverse function $\varphi^{-1}$ by
\begin{equation}
\varphi^{-1}(t)=\sup\{z:\varphi(z)=t\},
\quad
t\in[0,\nu_0(I)].
\label{eq5.1}
\end{equation}
Let $1\le i\le k$ be integers and
\begin{equation}
n=\sum^{k}_{i=1}n(i)
\label{eq5.2}
\end{equation}
be an arbitrary representation of $n$ as a sum of $k$ terms $n(i)\ge0$.
Define points $\lambda(0)=0<\lambda(1)\le\ldots\le\lambda(k)=1$ by
\begin{equation}
\lambda(j)=\varphi^{-1}\left(n^{-1}\sum^{j}_{i=1}n(i)\nu_0(I)\right),
\quad
1\le j\le k,
\label{eq5.3}
\end{equation}
and consider the intervals $\Delta(j)=[\lambda(j-1),\lambda(j)]\subset I$, $1\le j\le k$, of length $\ell(j)=\lambda(j)-\lambda(j-1)$.
The following statement is obvious.

\begin{lemma}\label{lem5.1}
The intervals $\Delta(j)$, $1\le j\le k$, form a partition of the unit interval~$I$,
\begin{equation}
I=\bigcup^{k}_{j=1}\Delta(j),
\quad
\sum^{k}_{j=1}\ell(j)=1,
\quad
\nu_0(\Delta(j_1)\cap\Delta(j_2))=0,
\quad j_1\ne j_2,
\label{eq5.4}
\end{equation}
and
\begin{equation}
\nu_0(\Delta(j))=\frac{n(j)}{n}\nu_0(I).
\label{eq5.5}
\end{equation}
\end{lemma}

If the measure $\nu_0\equiv0$ identically, then for any $n\ge1$ and any representation (5.2), the partition given in Lemma~\ref{lem5.1} takes the form
\begin{equation}
\Delta(1)=[0,1],
\quad
\Delta(j)=[1,1]=\{1\},
\quad
2\le j\le k.
\label{eq5.6}
\end{equation}
We agree that the partition \eqref{eq5.6} takes place also for $\nu_0\equiv0$ and $n=0$.

We now wish to extend Lemma~\ref{lem5.1} to the $d$-dimensional unit cube $I^d=[0,1]^d$.
Let $N\ge1$ be an integer, and let $k=\lceil N^{1/d}\rceil$, the smallest integer not less than $N^{1/d}$, so that $N\le k^d$.
Let
\begin{equation}
N=\sum^{k}_{i_1,\ldots,i_d=1}N(i_1,\ldots,i_d)
\label{eq5.7}
\end{equation}
be an arbitrary representation of $N$ as a sum of terms $N(i_1,\ldots,i_d)$ equal to $0$ or~$1$.
Introduce the non-negative integers
\begin{equation}
N(i_1,\ldots, i_q)=\sum^{k}_{i_{q+1},\ldots,i_d=1}N(i_1,\ldots,i_d),
\quad
1\le q< d.
\label{eq5.8}
\end{equation}
These integers satisfy the relations
\begin{equation}
N(i_1,\ldots,i_q)=\sum^{k}_{i_{q+1}=1}N(i_1,\ldots,i_{q+1}),
\quad
1\le q<d,
\label{eq5.9}
\end{equation}
and
\begin{equation}
N=\sum_{i_1=1}^kN(i_1).
\label{eq5.10}
\end{equation}

Let $\nu$ be a finite non-negative measure on $I^d$ with a continuous distribution function
\begin{equation}
\varphi(Z)=\nu([0,z_1]\times\ldots\times[0,z_d]),
\quad
Z=(z_1,\ldots,z_d)\in I^d.
\label{eq5.11}
\end{equation}

\begin{lemma}\label{lem5.2}
For any representation of an integer $N$ as a sum \eqref{eq5.7}, there exists a sequence of $d$ partitions $\PPP(q)$, $q=1,\ldots,d$, of the unit cube $I^d$ into rectangular boxes of the form
\begin{equation}
\Pi(i_1,\ldots,i_q)=\prod^{q}_{j=1}\Delta(i_1,\ldots,i_j)\times[0,1]^{d-q},
\quad
1\le i_j\le k,
\quad
1\le j\le q,
\label{eq5.12}
\end{equation}
where $\Delta(i_1,\ldots,i_j)\subset I$ are intervals of length $\ell(i_1,\ldots,i_j)$, such that the following hold:
\begin{itemize}
\item[(i)] If the indices $i_1,\ldots,i_{j-1}$ are fixed, then the intervals $\Delta(i_1,\ldots,i_{j-1},i_j)$, $i_j=1,\ldots,k$, form a partition of~$I$, and
\begin{equation}
\sum^{k}_{i_j=1}\ell(i_1,\ldots,i_{j-1},i_j)=1.
\label{eq5.13}
\end{equation}
\item[(ii)] The measures of the rectangular boxes \eqref{eq5.12} satisfy
\begin{equation}
\nu(\Pi(i_1,\ldots,i_q))=\frac{N(i_1,\ldots,i_q)}{N}\nu(I^d).
\label{eq5.14}
\end{equation}
\end{itemize}
\end{lemma}

\begin{proof}
We shall construct the sequence of partitions $\PPP(q)$, $q=1,\ldots,d$, by induction on~$q$. 

At the first stage, we define the partition $\PPP(1)$ as follows.
Consider the one-dimensional distribution function
\begin{equation}
\varphi(z_1)=\nu([0,z_1])\times[0,1]^{d-1})=\nu_0([0,z_1]),
\quad
z_1\in I,
\label{eq5.15}
\end{equation}
where $\nu_0$ is the corresponding measure on~$I$, $\nu_0(I)=\nu(I^d)$.
Putting $n=N$, $n(i_1)=N(i_1)$ in \eqref{eq5.2} with $i=i_1$, and applying Lemma~\ref{lem5.1} with the function \eqref{eq5.15}, we obtain a partition of $I$ into intervals $\Delta (i_1)$ of length $l(i_1)$, $i_1=1,\ldots,d$, with
\begin{displaymath}
\sum^{k}_{i_1=1}\ell(i_1)=1.
\end{displaymath}
Let $\Pi(i_1)=\Delta (i_1)\times[0,1]^{d-1}\subset I^d$.
Then
\begin{displaymath} 
\nu(\Pi(i_1))=\nu_0(\Delta (i_1))=\frac{N(i_1)}{N}\nu(I^d).
\end{displaymath}
The partition $\PPP(1)$ is now constructed.

Assume now that the partition $\PPP(q)$ is constructed already for some $q$, $1\le q<d$. 
Then we define the partition $\PPP(q+1)$ as follows.
For each rectangular box of the form \eqref{eq5.12}, we consider the one-dimensional distribution function
\begin{align}
\varphi(i_1,\ldots,i_q,z_{q+1})
&
=\nu\left(\prod^{q}_{j=1}\Delta(i_1,\ldots,i_j)\times[0,z_{q+1}]\times[0,1]^{d-q-1}\right)
\nonumber
\\
&
=\nu^{(i_1,\ldots,i_q)}_{0}([0,z_{q+1}]),
\quad
z_{q+1}\in I,
\label{eq5.16}
\end{align}
where $\nu^{(i_1,\ldots,i_q)}_{0}$ is the corresponding measure on~$I$, and
\begin{equation}
\nu^{(i_1,\ldots,i_q)}_{0}(I)=\nu(\Pi(i_1,\ldots,i_q))=\frac{N(i_1,\ldots, i_q)}{N}\nu(I^d).
\label{eq5.17}
\end{equation}
Putting $n=N(i_1,\ldots,i_q)$, $n(i_{q+1})=N(i_1,\ldots,i_{q},i_{q+1})$ in \eqref{eq5.2} 
with $i=i_{q+1}$, and applying Lemma~\ref{lem5.1} with the function \eqref{eq5.16}, 
we obtain a partition of $I$ into intervals
$\Delta(i_1,\ldots,i_q,i_{q+1})$, $1\le i_{q+1}\le k$, of length $\ell(i_1,\ldots,i_q,i_{q+1})$ such that
\begin{displaymath}
\sum^{k}_{i_{q+1}=1}\ell(i_1,\ldots,i_q,i_{q+1})=1.
\end{displaymath}
Let
\begin{equation}
\Pi(i_1,\ldots,i_{q+1})=\prod^{q+1}_{j=1}\Delta(i_1,\ldots,i_j)\times[0,1]^{d-q-1}.
\label{eq5.18}
\end{equation}
For these rectangular boxes, we have
\begin{equation}
\nu(\Pi(i_1,\ldots,i_{q+1}))=\frac{N(i_1,\ldots,i_{q+1})}{N}\nu(I^d).
\label{eq5.19}
\end{equation}
Indeed, if $N(i_1,\ldots,i_q)\ge1$, then in view of \eqref{eq5.17}, we have
\begin{displaymath}
\nu(\Pi(i_1,\ldots,i_{q+1}))
=\frac{N(i_1,\ldots,i_{q+1})}{N(i_1,\ldots,i_q)}\nu(\Pi(i_1,\ldots,i_q))
=\frac{N(i_1,\ldots,i_{q+1})}{N}\nu(I^d).
\end{displaymath}
If $N(i_1,\ldots,i_q)=0$, then $\nu(\Pi(i_1,\ldots,i_q))=0$, and the intervals $\Delta(i_1,\ldots,i_q,i_{q+1})$, $1\le i_{q+1}\le k$, are defined by \eqref{eq5.6}.
Hence $\nu(\Pi(i_1,\ldots,i_{q+1}))=0$ and \eqref{eq5.19} holds also. 
The partition $\PPP(q+1)$ is now constructed.
\end{proof}

Consider the partition $\PPP(d)$ constructed in Lemma~\ref{lem5.2}.
We have
\begin{equation}
\nu(\Pi(i_i,\ldots,i_d))=\frac{N(i_1,\ldots,i_d)}{N}\nu(I^d),
\label{eq5.20}
\end{equation}
where, by definition, the terms $N(i_1,\ldots,i_d)$ are equal to $0$ or~$1$, and the number of terms such that $N(i_1,\ldots,i_d)=1$ is equal to~$N$; see \eqref{eq5.7}.

Introduce the partition
\begin{equation} 
\PPP_N=\{\Pi(\alpha),\alpha\in\AAA\}
\label{eq5.21}
\end{equation}
of the unit cube, where
\begin{displaymath}
\AAA=\{\alpha=(i_1,\ldots,i_d):N(i_1,\ldots,i_d)=1\},
\quad
\#\AAA=N.
\end{displaymath}
We write
\begin{equation}
\Diam_1(\PPP_N)=\frac{1}{N}\sum_{\alpha\in\AAA}\diam(\Pi(\alpha),\Vert\cdot\Vert)
\label{eq5.22}
\end{equation}
for the average diameter of the partition $\PPP_N$ with respect to the Euclidean metric.

\begin{lemma}\label{lem5.3}
The partition $\PPP_N=\{\Pi(\alpha),\alpha\in\AAA\}$ is an equal-measure partition, so that
\begin{equation}
\nu(\Pi(\alpha))=N^{-1}\nu(I^d),
\quad
\alpha\in\AAA.
\label{eq5.23}
\end{equation}
Furthermore,
\begin{equation}
\Diam_1(\PPP_N)<d2^{d-1}N^{-1/d}.
\label{eq5.24}
\end{equation}
\end{lemma}

\begin{proof}
Note that \eqref{eq5.23} follows from \eqref{eq5.20} and the definition \eqref{eq5.21}. 

On the other hand, the Euclidean diameter of a rectangular box $\Pi(i_1,\ldots,i_d)$ does not exceed the sum of lengths of its sides.
In other words,
\begin{equation}
\diam(\Pi(i_1,\ldots,i_d),\Vert\cdot\Vert)\le\ell(i_1)+\ell(i_1,i_2)+\ldots+\ell(i_1,\ldots,i_d),
\label{eq5.25}
\end{equation}
where $\ell(i_1,\ldots,i_j)$, $1\le j\le d$, are the lengths of the interval $\Delta(i_1,\ldots,i_j)$; see \eqref{eq5.12}. 
Using \eqref{eq5.25}, \eqref{eq5.22} and \eqref{eq5.13}, we obtain
\begin{align}
N\Diam_1(\PPP_N)
&
\le\sum^{k}_{i_1,\ldots,i_d=1}\diam(\Pi(i_1,\ldots,i_d),\Vert\cdot\Vert)
\le\sum^{d}_{j=1}\sum^{k}_{i_i,\ldots,i_j=1}\ell(i_1,\ldots,i_j)
\nonumber
\\
&
=\sum^{d}_{j=1}k^{d-j}\sum^{k}_{i_1,\ldots,i_{j=1}}\ell(i_1,\ldots,i_j)
=\sum^{d}_{j=1}k^{d-1}
=dk^{d-1}.
\nonumber
\end{align}
Since $k=\lceil N^{1/d}\rceil$ and $k<N^{1/d}+1$, we have
\begin{displaymath}
N\Diam_1(\PPP_N)
<d(N^{1/d}+1)^{d-1}
=d(1+N^{-1/d})^{d-1}N^{1-1/d}
\le d2^{d-1}N^{1-1/d},
\end{displaymath}
and \eqref{eq5.24} follows.
\end{proof} 

\begin{proof}[Proof of Theorem~\ref{thm4.1}]
Let $\MMM$ be a $d$-rectifiable space.
Without loss of generality, we can assume in Definition~\ref{def1.1} that the measure $\nu$ is concentrated on the subset $\OOO\subset I^d$ and that both measures $\mu$ and $\nu$ are normalized so that $\mu(\MMM)=1$ and $\nu(\OOO)=1$.

Since the measure $\nu$ is absolutely continuous with respect to Lebesgue measure on~$I^d$, its distribution function \eqref{eq5.11} is continuous, so Lemmas \ref{lem5.2} and~\ref{lem5.3} can be applied.

For the measure~$\nu$, we consider the equal-measure partition $\PPP_N$ of the unit cube $I^d$ given in Lemma~\ref{lem5.3}. 
Consider the collection of subsets
\begin{equation}
\RRR_N =\{V(\alpha),\alpha\in\AAA\},
\quad
V(\alpha)=f(\Pi(\alpha)\cap\OOO),
\label{eq5.26}
\end{equation}
in the space $\MMM$.
Using \eqref{eq1.13}, we obtain
\begin{displaymath}
\mu(V(\alpha))=\nu(\Pi(\alpha)\cap\OOO)=\nu(\Pi(\alpha))=N^{-1},
\end{displaymath}
since the measure $\nu$ is concentrated on~$\OOO$. 
By definition, the map $f:\OOO\to\MMM$ is an injection.
Hence
\begin{displaymath}
\mu(V(\alpha_1)\cap V(\alpha_2))=\nu(\Pi(\alpha_1)\cap\Pi(\alpha_2)\cap\OOO)=0,
\quad
\alpha_1\ne\alpha_2.
\end{displaymath}
Thus the collection of subsets \eqref{eq5.26} is an equal-measure partition of~$\MMM$. 
By definition, the map $f:\OOO\to\MMM$ is also a Lipschitz map.
Therefore 
\begin{displaymath}
\diam(V(\alpha),\theta)
\le\Lip(f)\diam(\Pi(\alpha)\cap\OOO,\Vert\cdot\Vert)
\le\Lip(f)\diam(\Pi(\alpha),\Vert\cdot\Vert).
\end{displaymath}
In view of \eqref{eq3.2} and \eqref{eq5.22}, we have
\begin{equation}
\Diam_1(\RRR_N,\theta)\le\Lip(f)\Diam_1(\PPP_N).
\label{eq5.27}
\end{equation}
Substituting \eqref{eq5.24} into \eqref{eq5.27}, we obtain the desired bound \eqref{eq4.1}.
\end{proof}


\begin{thebibliography}{99} 

\bibitem{ref1} 
J.R. Alexander.
On the sum of distances between $n$ points on a sphere.
\textit{Acta Math. Hungar.}, 23 (1972), 443--448.

\bibitem{ref2} 
J.R. Alexander, J. Beck and W.W.L. Chen.
Geometric discrepancy theory and uniform distribution.
In \textit{Handbook of Discrete and Computational Geometry}, 2nd edition (J.E. Goodman and J. O'Rourke, eds.), pp 279--304
(Chapman \& Hall/CRC, 2004).

\bibitem{ref3} 
A. Barg and M.M. Skriganov.
Association schemes on general measure spaces and zero-dimensional Abelian groups.
\textit{Advances in Math.}, 281 (2015), 142--247.

\bibitem{ref4}
J. Beck.
Sums of distances between points on a sphere -- an application of the theory of irregularities of distribution to distance geometry.
\textit{Mathematika}, 31 (1984), 33--41.


\bibitem{ref5}
J. Beck and W.W.L. Chen.
\textit{Irregularities of Distribution} (Cambridge Tracts in Mathematics 89, Cambridge University Press, 1987).

\bibitem{ref6}
D. Bilyk, F. Dai and R. Matzke.
Stolarsky principle and energy optimization on the sphere.
\textit{Preprint}, (2016), available at https://arxiv.org/abs/1611.04420.

\bibitem{ref7}
J.S. Brauchart and J. Dick.
A simple proof of Stolarsky's invariance principle.
\textit{Proc. Amer. Math. Soc.}, 141 (2013), 2085--2096.

\bibitem{ref8}
S. Helgason.
\textit{Differential Geometry, Lie Groups, and Symmetric Spaces} (Academic Press, 1978).


\bibitem{ref9}
V.I. Levenshtein.
Universal bounds for codes and designs.
In \textit{Handbook of Coding Theory} (V.S. Pless and W.C. Huffman, eds.), pp 499--648
(Elsevier, 1998).

\bibitem{ref10}
P. Mattila.
\textit{Geometry of Sets and Measures in Euclidean Spaces. Fractals and Rectifiability} (Cambridge University Press, 1995).

\bibitem{ref11}
E.A. Rakhmanov, E.B. Saff and Y.M. Zhou.
Minimal discrete energy on the sphere.
\textit{Math. Res. Lett.}, 1 (1994), 647--662.

\bibitem{ref12}
M.M. Skriganov.
Point distributions in compact metric spaces, III. Two-point homogeneous spaces. 
\textit{Preprint POMI 7/2016}, (2016), in Russian, 
available at http://www.pdmi.ras.ru/preprint/2016/16-07.html.

\bibitem{ref13}
K.B. Stolarsky.
Sums of distances between points on a sphere, II.
\textit{Proc. Amer. Math. Soc.}, 41 (1973), 575--582.

\end{thebibliography}
\end{document}